\documentclass{amsart}
\usepackage{amssymb}
\usepackage{eucal}
\theoremstyle{plain}
\newtheorem{theorem}{Theorem}

\newtheorem{lemma}{Lemma}

\theoremstyle{definition}

\theoremstyle{remark}

\newtheorem{remark}{Remark}
\numberwithin{equation}{section}
\newcommand{\e}{\epsilon}

\newcommand{\de}{\delta}
\newcommand{\f}{\theta^{\xi,t}_\delta}
\newcommand{\R}{\mathbb R}
\newcommand{\N}{\mathbb N}

\newcommand{\C}{\mathbb C}

\newcommand{\Rn}{\mathbb R^n}

\newcommand{\Rm}{\mathbb R^{n+1}}
\newcommand{\sS}{\mathcal{S}}
\newcommand{\calP}{\mathcal{P}}

\begin{document}

\title[The Hardy Uncertainty Principle Revisited]{The Hardy Uncertainty Principle Revisited}
%%%%%%%%%%%%%%%%%%
%Author information
%%%%%%%%%%%%%%%%%%%%
\author{M. Cowling}
\address[M. Cowling]{School of Mathematics\\Watson Building\\
University of Birmingham\\Edgbaston\\ Birmingham B15 2TT.}
\email{M.G.Cowling@bham.ac.uk}
\thanks{The second and fifth authors are supported  by MEC grant, MTM2007-62186, the third and fourth authors by NSF grants DMS-0456583 and DMS-0456833 respectively}
%\thanks{}
%%%%%%%%%%%%%%%%%%%%
\author{L. Escauriaza}
\address[L. Escauriaza]{UPV/EHU\\Dpto. de Matem\'aticas\\Apto. 644, 48080 Bilbao, Spain.}
\email{luis.escauriaza@ehu.es}
%\thanks{}
%\thanks{}
%%%%%%%%%%%%%%%%%%%%%%
\author{C. E. Kenig}
\address[C. E. Kenig]{Department of Mathematics\\University of Chicago\\Chicago, Il. 60637 \\USA.}
\email{cek@math.uchicago.edu}
%\thanks{}
%%%%%%%%%%%%%%%%%%%%%%
\author{G. Ponce}
\address[G. Ponce]{Department of Mathematics\\
University of California\\
Santa Barbara, CA 93106\\
USA.}
\email{ponce@math.ucsb.edu}
%\thanks{}
%%%%%%%%%%%%%%%%%%%%%%
\author{L. Vega}
\address[L. Vega]{UPV/EHU\\Dpto. de Matem\'aticas\\Apto. 644, 48080 Bilbao, Spain.}
\email{luis.vega@ehu.es}
%\thanks{}
%%%%%%%%%%%%%%%%%%%%%%
\keywords{Fourier transform, Schr\"odinger evolutions}
\subjclass{Primary: 35B05. Secondary: 35B60}
%\date{}
%\dedicatory{}
%%%%%%%%%%%%%%
\begin{abstract}We give a real-variable proof of the Hardy uncertainty principle. The method is based on energy estimates for evolutions with positive viscosity, convexity properties of free waves with Gaussian decay at two different times, elliptic $L^2$-estimates and the invertibility of the Fourier transform on $L^2(\Rn)$ and $\mathcal S'(\Rn)$.
\end{abstract}
\maketitle
%%%%%%%%%%%%%%%%%%
\begin{section}{Introduction}\label{S: 1}
There are different ways of stating uncertainty principles for the Fourier transform:  a function $f$ and its Fourier
transform
\begin{equation*}
\hat f(\xi)=\frac 1{\sqrt {2\pi}}\int_{\R}e^{-i\xi\cdot x}f(x)\,dx,\  \xi\in\R,
\end{equation*} 
can not be highly concentrated  unless $f$ is zero. Among them one finds  the Hardy uncertainty principle ($A_1$) \cite{Har}  (see also \cite[pp.131]{StSh}),  its extension $(A_2)$ established in \cite{CoPr2} and the Beurling-H\"ormander result (B) in \cite{Hor2}:
\vspace{0.1cm}

$(A_1)$ \emph{If 
$f(x)=O(e^{-x^2/\beta^2})$, $\hat f(\xi)=O(e^{-4\xi^2/\alpha^2})$ and  $1/\alpha\beta>1/4$, then $f\equiv 0$. Also, if $1/\alpha\beta=1/4$, $f$ is a constant multiple of $e^{-x^2/\beta^2}$.}
\vspace{0.1cm}

$(A_2)$ \emph{ If $0<p,\,q\le \infty$, with at least one of them finite,
\begin{equation*}
\|e^{x^2/\beta^2}f\|_{L^p(\R)}+\|e^{4\xi^2/\alpha^2}\widehat f\|_{L^q(\R)}<+\infty,
\end{equation*}
and $1/\alpha\beta\ge  1/4$, then $f\equiv 0$.} 
\vspace{0.1cm}

$(B)$ \emph{ If $f$ is integrable on $\R$ and 
\begin{equation*}
\int_{\R} \int_{\R} |f(x)| |\widehat f(\xi)| e^{ |x\,\xi|}\,dx\,d\xi
\end{equation*}
 is finite, then $f\equiv 0$.}
 \vspace{0.1cm}

There has also been considerable interest in a better understanding of these results and on its extensions to higher dimensions, and to other settings (\cite{bonamie}, \cite{bonamie2}, \cite{M34}, \cite{SiSu} and \cite{psst}). We shall consider some of these questions at the end of section 3.

As far as we know the proofs of these results and its variants use  Complex Analysis: mainly the fact that multiplication of analytic functions is analytic and the Phragm\'en-Lindel\"of principle. 

The above results have equivalent formulations in terms of the unique solution in $C(\R,L^2(\Rn))$ of the Schr\"odinger evolution
\begin{equation}\label{E: IVP}
\begin{cases}
i\partial_tu+\triangle u=0,\ &\text{in}\ \R^{n+1},\\
u(0)=h,\ &\text{in}\ \Rn.
\end{cases}
\end{equation}
via the identity \eqref{E: equivalencia} below. In particular, when $n=1$ $(A_1)$, $(A_2)$ and $(B)$ are equivalent to
\vspace{0.1cm}

$(A_1')$ \emph{ If $u(x,0)=O(e^{-x^2/\beta^2})$, $u(x,T)=O(e^{-x^2/\alpha^2})$, $T>0$ and $T/\alpha \beta>1/4$, then $u\equiv 0$. Also, if $T/\alpha \beta = 1/4$, $u$ has initial data a constant multiple of $e^{-(1/\beta^2+i/4T)x^2}$.} 
\vspace{0.1cm}

 $(A_2')$  \emph{ If $1\le p,\,q\le \infty$, with at least one of them finite,
\begin{equation*}
\|e^{x^2/\beta^2}u(0)\|_{L^p(\R)}+\|e^{x^2/\alpha^2}u(T)\|_{L^q(\R)}<+\infty
\end{equation*}
and $T/\alpha\beta\ge 1/4$, then $u\equiv 0$.}
 \vspace{0.1cm}

 $(B')$ \emph{ If $u(0)$ is integrable and
 \begin{equation*}
 \int_{\R} \int_{\R} |u(x,0)| |u(\xi,T)| e^{ |x\,\xi|/2T}\,dx \,d\xi<+\infty,
 \end{equation*}
 then $u\equiv 0$.}
 \vspace{0.1cm}

Considering these results as a motivation, in \cite{ekpv06}, \cite{ekpv08b}, \cite{ekpv09}, \cite{ekpv09b}, and  \cite{kpv02} sufficient conditions on a  solution $u$ to \eqref{E: 1.111}, the potential $V$ and the behavior of the solution at two different times are sought in order to guarantee that $u$ is identically zero. The ideas developed in \cite{ekpv06} and \cite{ekpv08b}, were necessary in \cite{ekpv09}, to obtain an extension of the $L^2$-versions of the Hardy uncertainty principle in $(A_2)$ or $(A_2')$, with $p=q=2$. This extension is also valid for non-constant coefficient Schr\"odinger evolutions. In particular, it was shown in \cite{ekpv09}  that zero is the only solution $u$ in $C([0,T]),L^2(\Rn))$ to 
\begin{equation}\label{E: 1.111}\partial_tu=i\left(\triangle u+V(x,t)u\right),\ \text{in}\ \Rn\times [0,T],
\end{equation}
which verifies
\begin{equation}\label{E: una condicion}
\|e^{|x|^2/\beta^2}u(0)\|_{L^2(\Rn)}+\|e^{|x|^2/\alpha^2}u(T)\|_{L^2(\Rn)}<+\infty,
\end{equation}
when $n\ge 1$, $T/\alpha\beta > 1/4$, the potential $V$ is  bounded and either, $V(x,t)=V_1(x)+V_2(x,t)$, with $V_1$ real-valued and 
\begin{equation*}
\sup_{[0,T]}\|e^{T^2|x|^2/\left(\alpha t+\beta \left(T-t\right)\right)^2}V_2(t) \|_{L^\infty(\Rn)} < +\infty
\end{equation*}
or $\lim_{R\rightarrow +\infty}\|V\|_{L^1([0,T], L^\infty(\Rn\setminus B_R)}=0$.
\vspace{0.1cm}

The proof of this extension only uses \emph{real-variable}  methods and provides the first proof of an $L^2(\Rn)$ version of the Hardy uncertainty principle (up to the end-point case, $T/\alpha\beta=1/4$), which does not use Complex Analysis techniques. The methods in \cite{ekpv06}, \cite{ekpv08b} and \cite{ekpv09} are based on Carleman inequalities for certain evolutions. More precisely, on the convexity and log-convexity properties associated to the solutions of those evolutions. The Phragm\'en-Lindel\"of principle is replaced by convexity and log-convexity properties of appropriate $L^2$ quantities. Also notice that the product of log-convex functions is log-convex and this can be seen as the substitute to the fact that analyticity is preserved under multiplication.

Clearly, the methods based on Complex Analysis to prove the classical Hardy type uncertainty principles cannot handle or be adapted to establish uncertainty principles for solutions of evolutions with non constant coefficients. On the other hand, the methods in  \cite{ekpv06}, \cite{ekpv08b}, \cite{ekpv09} and \cite{ekpv09b} have shown to be successful with non constant evolutions, though the methods cannot reach, as they are understood up to \cite{ekpv09}, the end-point case for either, the $L^\infty(\Rn)$ or $L^2(\Rn)$ versions of the Hardy uncertainty principle. They correspond respectively to $1/\alpha\beta=1/4$ in Theorems \ref{T: 4}  and \ref{T: el teorema1}  below or to $T/\alpha\beta=1/4$ in $(A')$ and $(B')$ above, with $p=q=2$. Indeed, the following counterexample is given in \cite{ekpv09}:  whenever $T/\alpha\beta=1/4$ and $n\ge 1$, there is a time-dependent complex-valued potential $V$ with
\begin{equation*}
|V(x,t)|\lesssim\frac 1{1+|x|^2},\ \text{in}\ \Rn\times [0,T],
\end{equation*}
and such that \eqref{E: 1.111} has a $C^\infty(\Rn\times [0,T])$ nonzero solution $u$ verifying \eqref{E: una condicion}.  Of course, $u$ also verifies \eqref{E: una condicion} if one replaces the $L^2(\Rn)$-norm in \eqref{E: una condicion} by the $L^\infty(\Rn)$-norm. Hence, the methods in \cite{ekpv09} must be modified if one seeks a \emph{real variable} proof of the end-point case.

In this work we find such a modification of the arguments in \cite{ekpv09} and provide a new proof of either the $L^2(\Rn)$ ($p=2=q$ in $B$ and $B'$) or $L^\infty(\Rn)$ ($A$ and $A'$) versions of the Hardy uncertainty principle. The modification also avoids complex methods. In particular,  we first prove with \emph{real-variable} techniques the following $L^2(\Rn)$ version of the Hardy uncertainty principle.
 
\begin{theorem}\label{T: el teorema1} Assume that $h:\Rn\longrightarrow \R$, $n\ge 1$, verifies
\begin{equation*}
\|e^{|x|^2/\beta^2}h\|_{L^2(\Rn)}+\|e^{4|\xi|^2/\alpha^2}\widehat h\|_{L^2(\Rn)}<+\infty
\end{equation*}
and $1/\alpha\beta\ge  1/4$, then $h\equiv 0$.
\end{theorem}

\noindent Then, we prove Theorem \ref{T: 4}, the classical Hardy uncertainty principle, with \emph{real-variable} methods and as a Corollary of Theorem \ref{T: el teorema1} for $n=1$.
\begin{theorem}\label{T: 4}
Assume that  $n\ge 1$, $1/\alpha\beta>1/4$ and $h:\Rn\longrightarrow \R$ verifies
\begin{equation*}
\|e^{|x|^2/\beta^2}h\|_{L^\infty(\Rn)}+\|e^{4|\xi|^2/\alpha^2}\widehat h\|_{L^\infty(\Rn)}<+\infty,
\end{equation*}
then $h\equiv 0$. Also, if $1/\alpha\beta=1/4$, $h$  is a constant multiple of $e^{-|x|^2/\beta^2}$.
\end{theorem}
Our  proof uses Theorem 3 in \cite{ekpv09} (See Lemma \ref{L: lamejora} below for the version we need here). It is related to the interior \emph{improvement} of the Gaussian decay of a free wave which has Gaussian decay at two different times, and proved with \emph{real-variable} methods in  \cite{ekpv09}, see also \cite{ekpv08a}. 

The outline of our proof is as follows. When $h$ satisfies the hypothesis in Theorem \ref{T: el teorema1}, we may assume $\alpha=\beta=2$ and if $u$ is the free wave verifying \eqref {E: IVP}, Theorem 3 in  \cite{ekpv09} (see also Lemma \ref{L: lamejora} below) implies that
\begin{equation}\label{E: la segunda}
\sup_{\R}\|e^{\frac{|x|^2}{4(1+t^2)}}u(t)\|_{L^2(\Rn)}<+\infty.
\end{equation}
Then, we define $G(x,t)=(t-i)^{-\frac n2}e^{-\frac{|x|^2}{4i(t-i)}}$, the free wave whose precise Gaussian decay is 
\begin{equation*}
\left(1+t^2\right)^{-n/4}\, e^{-|x|^2/4(1+t^2)}.
\end{equation*} 
Set $g=u/G$ and observe that $\phi(\xi,t)=\widehat g(\xi,t)$, the Fourier transform of $g(t)$, verifies
\begin{equation}\label{E: otra ecuacion}
\partial_t\phi-\frac \xi{t-i}\,\cdot\nabla\phi-\frac n{t-i}\, \phi+i|\xi|^2\phi=0,\ \text{in}\ \Rm
\end{equation}
and
\begin{equation}\label{E: otra condici—n}
\sup_{\R}\frac{\ \ \ \|\phi(t)\|_{L^2(\Rn)}}{(1+t^2)^{\frac n4}}<+\infty\,.
\end{equation}
Then, we use $L^2$ elliptic estimates to justify the calculations leading to a certain log-convexity property of solutions to \eqref{E: otra ecuacion}. This convexity property implies the Liouville result in Theorem \ref{T: Theroem 3} below and Theorem \ref{T: el teorema1} follows.
\begin{theorem}\label{T: Theroem 3}
Zero is the only weak solution to  \eqref{E: otra ecuacion} verifying \eqref{E: otra condici—n}.
\end{theorem}
Altogether, our proof uses energy estimates for evolutions with positive viscosity \cite[Lemma 1]{ekpv08b}, the Gaussian convexity properties and the improvement of the Gaussian decay of free waves with Gaussian decay at two different times obtained in \cite[Lemma 3]{ekpv08b}, and \cite[Theorem 3]{ekpv09} respectively (see also Lemma \ref{L: lamejora} below), $L^2$ elliptic estimates, and the invertibility of the Fourier transform on $L^2(\Rn)$ and in the class of tempered distribution $\mathcal S'(\Rn)$.

We also use the formula
\begin{equation}\label{E: la formulita}
\begin{aligned}
u(x,t)=(2\pi)^{-\frac n2}\int_\R &e^{ix\xi-i\xi^2t}\widehat h(\xi)\,d\xi\\&=(4\pi |t|)^{-\frac n2}e^{\pi i n\,\text{sgn}\,t/2}\int_\R e^{-|x-y|^2/4it}h(y)\,dy,
\end{aligned}
\end{equation}
for the solution $u$ to \eqref{E: IVP}, where
 \begin{equation*}
\ \hat h(\xi)=(2\pi)^{-\frac n2}\int_{\R}e^{-i\xi\cdot x}h(x)\,dx,\ \xi\in\Rn,
\end{equation*}
is the Fourier transform of $h$. Expanding the square in the second integral in \eqref{E: la formulita}, $u$ can be written  as
\begin{equation}\label{E: equivalencia}
u(x,t)
= (2  |t|)^{-\frac n2}e^{\pi in\,\text{sgn}\,t/2}e^{ix^2/4t}\,
\widehat{e^{\frac{i|\,\cdot\,|^2}{4t}}h}\left(\frac{x}{2 t}\right).
\end{equation}
Notice that in the above identities we use that 
\begin{equation*}
\int_{-\infty}^{+\infty}\, e^{\pi ix^2}\,dx=e^{\pi i/4}.
\end{equation*}
We recall that this formula and the invertibility of the Fourier transform  on $L^2(\Rn)$ and in the class of tempered distribution $\mathcal S'(\Rn)$ can be established with real-variable methods. Finally, we also use in section \ref{S: 3} that the Fourier transform of  the principal value distribution
\begin{equation*}
\left(\text{p.v.}\, \tfrac 1x\, , \varphi\right)=\lim_{\e\rightarrow 0}\int_{|x|> \e}\tfrac{\varphi(x)} x\,dx,\quad \varphi\in C_0^\infty(\R)
\end{equation*}
and $e^{-x^2/4}$ are respectively constant multiples of the sign function and $e^{-\xi^2}$. We remark that the latter and the second identity in \eqref{E: la formulita} can all be obtained with real variable methods.

Though our proof of Theorems \ref{T: el teorema1} does not use \emph{analytic functions},  we note that the inspiration for the convexity arguments used to prove Theorem \ref{T: Theroem 3} comes from the following formal fact: 

\vspace{0.1cm}
\emph{Assume that $\phi$ verifies the conditions in Theorem \ref{T: Theroem 3} for $n=1$ and set
\begin{equation*}
J(z)=\frac{x}{x+iy}\,e^{-x^2-ixy}\, \phi(x,y/x),\ \text{for}\ z=x+iy,\  x\neq 0\  \text{and}\ y\in\R,
\end{equation*}
then, $J$ is analytic in the open right half-plane and
\begin{equation*}
\sup_{-\frac\pi 2<\theta <\frac\pi 2}\int_0^{+\infty}e^{2x^2}|J(re^{i\theta})|^2\,dr<+\infty.
\end{equation*}}

That such an analytic function is zero can be established using Carleman inequalities for the Cauchy-Riemann operator
\begin{equation*}
\partial_{\overline z}=\tfrac 12\left(\partial_x+i\partial_y\right),
\end{equation*}
and the standard proof of these Carleman inequalities relies on the following fact:

\vspace{0.2cm}
\emph{For a given smooth function $\varphi: \Omega\subset\R^2\longrightarrow\R$, write 
\begin{equation*}
e^{\varphi}\,\partial_{\overline z}\,e^{-\varphi}=\mathcal S+\mathcal A,
\end{equation*}
where $\mathcal S$ and $\mathcal A$ are symmetric and skew-symmetric operators on $C_0^\infty(\Omega)$. Then,  $[\mathcal S,\mathcal A]$, the commutator of $\mathcal S$ and $\mathcal A$, verifies
\begin{equation*}
[\mathcal S,\mathcal A]=\tfrac 14\,\triangle\varphi.
\end{equation*}}
Our proof consists in carrying out  these ideas in the original coordinates $(\xi,t)$ of $\phi$.
\end{section}

\begin{section}{A real-variable approach}\label{S: 2}
In the sequel $N_{A,\,\dots}$ denotes a constant depending on the variable $A$ and the other posible variables in the subscript.  $\|f\|$ denotes the $L^2$-norm of $f$ over the Euclidean space where it is defined.

In the proof of Theorem \ref{T: el teorema1} we need Lemmas \ref{L: lamejora} and \ref{L: freq1} below. The first follows from  \cite[Theorem 3]{ekpv09} and the second from \cite[Lemmas 1 and 2]{ekpv09}.

\begin{lemma}\label{L: lamejora}
Assume that $u$ in $C([0,T], L^2(\Rn))$ verifies
\begin{equation*}
\partial_tu+i\triangle u=0,\ \text{in}\ \Rn\times [0,T],
\end{equation*}
$n\ge 1$, $\alpha$ and $\beta$ are positive and $T/\alpha\beta\le 1/4$. Then,

\begin{equation*}
\sup_{[0,T]}\|e^{a(t)|x|^2}u(t)\|
\le\|e^{|x|^2/\beta^2}u(0)\|+ \|e^{|x|^2/\alpha^2}u(T)\|,
\end{equation*}
where
\[a(t)=\frac {\alpha\beta RT}{2\left(\alpha t+\beta (T-t)\right)^2+2R^2\left(\alpha t - \beta (T-t)\right)^2}\]
and $R$ is the smallest root of the equation 
\begin{equation*}
\frac T{\alpha\beta}=\frac R{2\left(1+R^2\right)}\, .
\end{equation*}
\end{lemma}

\begin{remark}\label{R:2} $1/a(t)$ is convex and attains its minimum value in the interior of $[0,T]$, when
\begin{equation}\label{E: segunda condi}
|\alpha-\beta|<R^2\left(\alpha+\beta\right).
\end{equation}
Thus, both $u(0)$ and $u(T)$ are generated by waves with faster Gaussian decay in $(0,T)$, when \eqref{E: segunda condi} holds.
\end{remark}

Lemma \ref{L: freq1} is used to justify the validity of a formal log-convexity property of solutions to \eqref{E: otra ecuacion}
.
\begin{lemma}\label{L: freq1}
Let $\lambda\in\Rn$, $T>0$, $0<\e\le 1$ and $f=f(\xi,t)$ in $C^\infty(\R,\mathcal S(\Rn))$ verify
\begin{equation*}
\partial_tf-\frac\xi{t-i}\cdot\nabla f-\frac{n-\lambda\cdot\xi}{t-i}\,f=F ,\ \text{in}\ \Rm.
\end{equation*}
Set $H(t)=\|f(t)\|^2$. Then,
\begin{equation}\label{E: una propiedad de convexidad logar'tmica}
\frac{H(t)+\e}{\,\left(1+t^2\right)^{\frac n2}}\le \left(\frac{H(-T)+\e}{\,\left(1+T^2\right)^{\frac n2}}\right)^{\theta(t,T)}\left(\frac{H(T)+\e}{\,\left(1+T^2\right)^{\frac n2}}\right)^{1-\theta(t,T)}e^{N_{T,\e}},
\end{equation}
when $-T\le t\le T$, 
\begin{equation*}
\theta(t,T)=\frac{\arctan T-\arctan t}{2\arctan T}\ \ \text{and}\ \ N_{T,\e}\le \e+\tfrac T{\e^2}\|F\|_{L^2(\R\times [-T,T])}^2.
\end{equation*}
\end{lemma}
\begin{proof}[Proof of Lemma \ref{L: freq1}]
The proof is based on the following facts:
\begin{equation}\label{E: simetrico y antisimetrico}
\begin{aligned}
&\frac\xi{t-i}\cdot\nabla+\frac{n-\lambda\cdot\xi}{t-i}=\mathcal S+\mathcal A,\\
&\mathcal S=\frac i{1+t^2}\left(\xi\cdot\nabla+\frac n2\right)-\frac{t\,\lambda\cdot\xi}{1+t^2}+\frac{nt}{2(1+t^2)}\, ,\\
&\mathcal A=\frac t{1+t^2}\left(\xi\cdot\nabla +\frac n2\right)-\frac{i\lambda\cdot\xi }{1+t^2}+\frac{in}{2(1+t^2)}\ ,
\end{aligned}
\end{equation}
$\mathcal S$ and $\mathcal A$ are respectively symmetric and skew-symmetric operators on $\mathcal S(\R)$,
\begin{equation}
\begin{aligned}\label{E: otrassuperformuala}
&\left[\mathcal S,\mathcal A\right]=\frac{\lambda\cdot\xi}{1+t^2}\\
&\mathcal S_t=-\frac{2it}{(1+t^2)^2}\left(\xi\cdot\nabla+\frac n2\right)-\frac{(1-t^2)\lambda\cdot\xi}{\,\,(1+t^2)^2}+\frac{n(1-t^2)}{\,\,\, 2(1+t^2)^2},\\
&\mathcal S_t+[\mathcal S,\mathcal A]=-\frac{2it}{(1+t^2)^2}\left(\xi\cdot\nabla+\frac n2\right)+\frac{2t^2\,\lambda\cdot\xi}{(1+t^2)^2}+\frac{n(1-t^2)}{\,\,\,2(1+t^2)^2}\, ,\
\end{aligned}
\end{equation}
where $\left[\mathcal S,\mathcal A\right]$ and $\mathcal S_t$ denote respectively the space-commutator of $\mathcal S$ and $\mathcal A$ and the time-derivative operator of $\mathcal S$. Moreover,
\begin{equation}\label{E: la cuarta}
(1+t^2)\,\mathcal S_t+(1+t^2)\, [\mathcal S,\mathcal A]+2t\,\mathcal S=\frac n2\,,
\end{equation}
and Lemma \ref{L: freq1} follows from \eqref{E: simetrico y antisimetrico}, \eqref{E: otrassuperformuala}, \eqref{E: la cuarta} and Lemmas 1 and 2 in \cite{ekpv09} with $\psi\equiv - n/2$.
\end{proof}
\begin{remark}\label{R:1} The main idea behind  Lemma \ref{L: freq1} is that \eqref{E: la cuarta} implies that
\begin{equation}\label{E: loglog}
\partial_t\left(\left(1+t^2\right)\,\partial_t\log{\left(H+\e\right)}\right)\\\ge n,
\end{equation}
when $\e>0$ and $F\equiv 0$ and \eqref{E: una propiedad de convexidad logar'tmica} is the log-convexity property associated to \eqref{E: loglog}.
\end{remark}

\begin{proof}[Proof of Theorem \ref{T: el teorema1}]As we already said the case $1/(\alpha\beta)>1/4$ was proved in \cite{ekpv09} by real variable methods. For the remaining case we can assume by rescaling that $\alpha=\beta=2$ and
\begin{equation}\label{E: simplificacion}
\|e^{|x|^2/4}h\|+\|e^{|\xi|^2}\,\widehat h\|<+\infty\, .
\end{equation}
Let $u$ be the solution to \eqref{E: IVP}. From  \eqref{E: simplificacion}, $u$ is in $C^\infty(\R,H^\infty(\Rn))$. Define,
\begin{equation}\label{E: la transformacion}
v(x,t)=\left(it\right)^{-\frac n2}e^{-x^2/4it}\overline{u}(x/t, 1/t -1),\ \text{when}\ 0\le t\le 1.
\end{equation}
The first formula for $u$ in  \eqref{E: la formulita}  and \eqref{E: la transformacion} give
\begin{equation*}
v(x,t)=\left(4\pi i t\right)^{-\frac n2}\int_\R e^{-|x-\xi|^2/4it}\,2^{-\frac n2}e^{-i\xi^2/4}\,\overline{\widehat h}(\xi/ 2)\,d\xi\, .
\end{equation*} 
Thus, $v$ verifies
\begin{equation*}
\begin{cases}
i\partial_tv+\triangle v=0,\ &\text{in}\ \Rn\times (0,+\infty),\\
v(x,0)=2^{-\frac n2}e^{-i|x|^2/4}\,\overline{\widehat h}(x/2),\ &\text{in}\ \Rn,
\end{cases}
\end{equation*}
it is in $C^\infty(\R,H^\infty(\Rn))$ and $v(x,1)=i^{-\frac n2}e^{-x^2/4i}\,\overline h(x)$. These facts, and \eqref{E: simplificacion} show that
\begin{equation*}
\|e^{|x|^2/4}v(0)\|+\|e^{|x|^2/4}v(1)\|<+\infty.
\end{equation*}
From Lemma \ref{L: lamejora} with $\alpha=\beta=2$ and $T=1$, 
\begin{equation*}
\sup_{[0,1]} \|e^{\frac{|x|^2}{4\left(1-2t(1-t)\right)}}v(t)\|<+\infty
\end{equation*}
and undoing the changes of variables in \eqref{E: la transformacion},
\begin{equation*}
\sup_{[0,+\infty)}\|e^{\frac{|x|^2}{4(1+t^2)}}u(t)\|<+\infty\, .
\end{equation*}
Applying the same reasoning to $\overline u(-t)$, we get
\begin{equation*}%\label{E: la segunda}
\sup_{\R}\|e^{\frac{|x|^2}{4(1+t^2)}}u(t)\|<+\infty\, .
\end{equation*}
Define,
\begin{equation*}
G(x,t)=(t-i)^{-\frac n2}e^{-\frac {|x|^2}{4i(t-i)}}=(t-i)^{-\frac n2}e^{-\frac{(1-it)}{4(t^2+1)}|x|^2}\  \text{and}\  g=u/G.
\end{equation*}
 Then,
\begin{equation*}
\partial_tg-i\triangle g+\frac x{t-i}\cdot\nabla g=0,\ \text{in}\ \Rm\ ,\ \sup_{\R}\frac{\|g(t)\|}{(1+t^2)^{\frac n4}}<+\infty\, .
\end{equation*}
Setting, $\phi(\xi,t)=\widehat g(\xi,t)$, the Fourier transform of $g(t)$, $\phi$ verifies \eqref{E: otra condici—n}, it solves \eqref{E: otra ecuacion} in the distribution sense and Theorem \ref{T: el teorema1} follows from Theorem \ref{T: Theroem 3}.
\end{proof}
\begin{proof}[Proof of Theorem \ref{T: Theroem 3}] Let $\phi$ verify  \eqref{E: otra condici—n} and solve \eqref{E: otra ecuacion} in the distribution sense. Define 
\begin{equation}\label{E: quejodidadefinicion}
f(\xi,t)=e^{\lambda\cdot\xi-\frac{|\xi|^2(1+t^2)}2}\phi(\xi,t),
\end{equation}
when $\lambda\in\Rn$. $f$ verifies
\begin{equation}\label{E: una ecuacion simplificada}
\partial_tf-\frac\xi{t-i}\cdot\nabla f-\frac{n-\lambda\cdot\xi}{t-i}\,f=0 ,\ \text{in}\ \Rm
\end{equation}
in the sense of distributions. Formally, Lemma \ref{L: freq1} and  Remark \ref{R:1} give that
\begin{equation*}
\mathcal H(t)=\frac{\|f(t)\|^2}{\,(1+t^2)^{\frac n2}}=\frac{\|e^{\lambda\cdot\xi-\frac{|\xi|^2(1+t^2)}2}\phi(t)\|^2}{\,(1+t^2)^{\frac n2}}
\end{equation*}
verifies
\begin{equation*}
\partial_t\left((1+t^2)\partial_t\log{\mathcal H(t)}\right)\ge 0,\ \text{in}\ \R
\end{equation*}
and
\begin{equation}\label{E: la convexidad}
\mathcal H(t)\le \mathcal H(-T)^{\theta(t,T)}\mathcal H(T)^{1-\theta(t,T)},\ \text{when}\ -T\le t\le T.
\end{equation}
Because 
\begin{equation*}
\mathcal H(\pm T)\le e^{\frac{|\lambda|^2}{1+T^2}}\sup_\R{\frac{\|\phi(t)\|^2}{\,(1+t^2)^{\frac n2}}},
\end{equation*}
we get
\begin{equation}\label{E: lotercero}
\frac{\|e^{\lambda\cdot\xi-\frac{|\xi|^2(1+t^2)}2}\phi(t)\|^2}{\,(1+t^2)^{\frac n2}}\le e^{\frac{|\lambda|^2}{1+T^2}}\sup_\R\frac{\|\phi(t)\|^2}{\,(1+t^2)^{\frac n2}},
\end{equation}
when $-T\le t\le T$. Letting $T$ tend to $+\infty$ in \eqref{E: lotercero}
\begin{equation*}
\sup_{\R}{\frac{\|e^{\lambda\xi-\frac{|\xi|^2(1+t^2)}2}\phi(t)\|^2}{\,(1+t^2)^{\frac n2}}}\le \sup_{\R}{\frac{\|\phi(t)\|^2}{(1+t^2)^{\frac n2}}}\, ,
\end{equation*}
 which implies $\phi\equiv 0$, after letting $|\lambda|$ tend to infinity.
 
 To finish the proof we must show that the above claims are correct. In particular, suffices to show that $\phi$ verifies \eqref{E: la convexidad}, when $T>0$ and $t=0$. This can be done with similar arguments to the ones used in \cite[Lemma 4]{ekpv09} and for the reader's convenience we include them here. 
 
 The equation \eqref{E: una ecuacion simplificada} can be written as
\begin{equation*}
\partial_tf-\nabla\cdot\left(\frac\xi{t-i}\, f\right)+\frac{\lambda\cdot\xi}{t-i}\,f=0,\ \text{in}\ \Rm
\end{equation*}
and $f$ verifies
\begin{equation}\label{E: soluciondebil}
\int f(y,s)\left(-\partial_s\Theta+\frac y{s-i}\cdot\nabla\Theta+\frac{\lambda\cdot y}{s-i}\,\Theta\right)\,dyds=0,\ \text{for all}\ \Theta\in C^\infty_0(\Rm,\R).
\end{equation}

 Let $\theta$ in $C^\infty_0(\Rm)$ be a standard mollifier supported in the unit ball of $\Rm$ and for $0<\de\le \frac 14$, set $f_{\delta}(\xi,t)=f\ast\theta_\delta(\xi,t)$ and
\begin{equation*}
\theta_\de(y,s)=\theta^{\xi,t}_\de(y,s)=\delta^{-n-1}\theta(\tfrac{\xi-y}\delta\, ,\tfrac{t-s}\delta)\, .
 \end{equation*}
Then, $f_\de$ is in $C^\infty(\R,\mathcal S(\Rn))$ and
\begin{equation}\label{E: formulalarga}
\begin{aligned}
&\left(\partial_tf_\de-\frac\xi{t-i}\cdot\nabla f_\de-\frac{n-\lambda\cdot\xi}{t-i}\,f_\de \right)(\xi,t)\\
&=\int  f\left[-\partial_s\theta^{\xi,t}_\de+\frac y{s-i}\cdot\nabla\theta^{\xi,t}_\de+\frac{\lambda \cdot y}{s-i}\,\theta^{\xi,t}_\de\right]dyds\\
&+\int f\left[\frac{\lambda\cdot\xi}{t-i}-\frac{\lambda\cdot y}{s-i}\right]\f dyds\\
&+\int f\left[\frac\xi{t-i}-\frac y{s-i}\right]\cdot\nabla\f dyds-\frac n{t-i}\int f\,\f dyds.
\end{aligned}
\end{equation}
The last identity and \eqref{E: soluciondebil} give
\begin{equation}\label{E: controldelerror}
\partial_tf_\de-\frac\xi{t-i}\,\cdot\nabla f_\de-\frac{n-\lambda\cdot\xi}{t-i}\,f_\de=F_\de,
\end{equation}
where $F_{\de}$ is the sum of the last three integrals in \eqref{E: formulalarga}. Moreover, there is  $N_{\lambda}$ such that
\begin{equation}\label{E: otras desigualdades}
|f_\de(\xi,t)|+|\xi f_\de(\xi,t)|+|F_{\de}(\xi,t)|\le  \frac{N_{\lambda}}{\delta^{n+1}}\int_{t-\de}^{t+\de}\int_{B_\delta(\xi)}|\phi|\,dyds,
\end{equation}
when $(\xi,t)\in \Rm$.
From \eqref{E: controldelerror}
\begin{equation*}
(t-i)\partial_tf_\de-\xi\cdot\nabla f_\de=(n-\lambda\cdot\xi)f_\de+(t-i)F_\de
\end{equation*}
and
\begin{equation}\label{E: que jodida formula}
\int\varphi^2(t)|(t-i)\partial_tf_\de-\xi\cdot\nabla f_\de|^2\,d\xi dt=\int\varphi^2(t)|(n-\lambda\cdot\xi)f_\de+(t-i)F_\de|^2\,d\xi dt,
\end{equation}
when $\varphi\in C_0^\infty(\R)$. From \eqref{E: que jodida formula}, the identity
\begin{equation*}
|(t-i)\partial_tf_\de-\xi\cdot\nabla f_\de|^2=|\partial_tf_\de|^2+|t\partial_tf_\de-\xi\cdot\nabla f_\de|^2+i\left(\partial_tf_\de\overline{\xi\cdot\nabla f_\de}-\xi\cdot\nabla f_\de\overline{\partial_tf_\de}\right)
\end{equation*}
and integration by parts, we find that
\begin{equation}\label{E: otra jodidisima formulaII}
\begin{aligned}
&\int\varphi^2\left[|\partial_tf_\de|^2+|t\partial_tf_\de-\xi\cdot\nabla f_\de|^2\right]\,d\xi dt=\int\varphi^2|(n-\lambda\cdot\xi)f_\de+(t-i)F_\de|^2\,d\xi dt\\
&+\int i\varphi\left(2t\varphi '-n\varphi\right)f_\de\overline{\partial_tf_\de}\,d\xi dt+\int 2i\varphi\varphi' f_\de\overline{\left(\xi\cdot\nabla f_\de-t\partial_tf_\de\right)}\,d\xi dt,
\end{aligned}
\end{equation}
and \eqref{E: otra jodidisima formulaII}, \eqref {E: otras desigualdades}, \eqref{E: quejodidadefinicion}
 and \eqref{E: controldelerror} imply that
\begin{equation}\label{E: fiolla}
\|\xi\cdot\nabla f_\de\|_{L^2(\Rn\times [-T,T])}\le N_{T,\lambda}\,\sup_\R\frac{\|\phi(t)\|}{\sqrt[4]{1+t^2}}.
\end{equation}
Clearly, \eqref{E: fiolla} holds with $f$ replacing $f_\de$ and the last two integrals in \eqref{E: formulalarga} can be written as
\begin{equation*}
-\int\left(y\cdot\nabla f+nf\right)\left[\frac 1{t-i}-\frac 1{s-i}\right]\f dyds-\frac 1{t-i}\,f\ast\eta_\de(\xi,t),
\end{equation*}
where 
\begin{equation*}
\eta(\xi,t)=\xi\cdot\nabla\theta(\xi,t)+n\,\theta(\xi,t),\ \eta_\de(\xi,t)=\delta^{-n-1}\eta(\tfrac\xi\de,\tfrac\xi t),
\end{equation*}
 is a mollifier verifying
\begin{equation}\label{E: increible folla}
\int\eta(y,s)\,dy=0,\ \text{for all}\ s\in\R.
\end{equation}
Altogether,
\begin{align*}
F_\de(\xi,t)&=\int f\left[\frac{\lambda\cdot\xi}{t-i}-\frac{\lambda\cdot y}{s-i}\right]\f dyds\\
&-\int\left(y\cdot\nabla f+nf\right)\left[\frac 1{t-i}-\frac 1{s-i}\right]\f dyds-\frac 1{t-i}\,f\ast\eta_\de(\xi,t),
\end{align*}
and now it is simple to verify that
\begin{equation}\label{E: algo fundamen}
\lim_{\delta\to 0}\|F_\de\|_{L^2(\Rn\times [-T,T])}=0,\  \text{when}\ T>0\ \text{and}\ f\in C(\R,L^2(\Rn)).
\end{equation}
Apply now Lemma \ref{L: freq1} to $f_\de$ and recall \eqref{E: algo fundamen}. We get
\begin{equation}\label{E: acotacionfinal}
H_\de(0)\le\frac 1{\,(1+T^2)^{\frac n2}}\left(H_\de(-T)+\e\right)^{\frac 12}\left(H_\de(T)+\e\right)^{\frac 12}e^{N_{T,\de,\e}},\ \text{when}\ \e>0,
\end{equation}
with $H_\de(t)=\|f_\de(t)\|^2$ and
\begin{equation*}
N_{T,\de,\e}\le \e+\tfrac T{\e^2}\|F_\de\|_{L^2(\Rn\times [-T.T])}^2.
\end{equation*}
Letting then $\de$ and $\e$ tend to zero in \eqref{E: acotacionfinal}, \eqref{E: la convexidad} follows for $t=0$.
\end{proof}
\begin{remark}\label{R: Beltrami}
 According to \cite[Chapter 7]{atg} the equation \eqref{E: una ecuacion simplificada} is equivalent for $n=1$ to a degenerate first order elliptic system for the real and imaginary parts of $f$. The system ceases to be elliptic either on the line $\xi=0$, or when $\xi$ is large. Therefore, the elliptic theory implies that $f\in C^\infty(\R^2\setminus\{(\xi,t): \xi=0\})$. Here, we need global estimates for solutions to \eqref{E: una ecuacion simplificada} and this is the reason why we use \eqref{E: otra jodidisima formulaII} and the mollifiers. 
 \end{remark}
  \end{section}
  
 %&&&&

 \begin{section}{The $L^\infty(\Rn)$ version follows from the $L^2(\R)$ one  and related issues}\label{S: 3} 
\begin{proof}[Proof of Theorem \ref{T: 4}] It suffices to prove Theorem \ref{T: 4} when $\alpha=\beta=2$ and
\begin{equation}\label{E: la condicion en Linf}
\|e^{|x|^2/4}h\|_{L^2(\Rn)}+\|e^{|\xi|^2}\widehat h\|_{L^2(\Rn)}<+\infty.
\end{equation}
 Assume first $n=1$. Clearly, $h$ is smooth and if
 \begin{equation}\label{E: mas peque–a formula}
 g(x) = h(x) - h(0) e^{-x^2/4},
 \end{equation}
$g$ also verifies \eqref{E: la condicion en Linf} and is
smooth.  Now consider, 
\begin{equation}\label{E: una peque–a definicion}
f(x) = g(x)/x.
\end{equation}
Obviously, $|f(x)| \le
e^{-x^2/4}/|x|$ and $f$ is continuous in $[-1,1]$.  The Fourier transform of $f$ is, up to a multiple, the convolution of the sign function with $\widehat{g}$. Because $g(0) = 0$, we have
\begin{equation}\label{E: una f simpatica}
\begin{aligned}
\int_{\R} \text{sgn}(\xi - \eta) \widehat{g}(\eta)\, d\eta
&= \int_{\eta < \xi} \widehat{g}(\eta)\, d\eta - \int_{\eta > \xi}
\widehat{g}(\eta)\, d\eta
\\ &= 2\int_{\eta < \xi} \widehat{g}(\eta)\, d\eta - \int_{\R}
\widehat{g}(\eta)\, d\eta
= 2\int_{\eta < \xi} \widehat{g}(\eta)\, d\eta
\end{aligned}
\end{equation}
and there is a similar expression when $\xi >
0$.  If we estimate \eqref{E: una f simpatica} for $\xi <0$, we get
\begin{equation*}
|\int_{\eta < \xi} \widehat{g}(\eta)\, d\eta |
\lesssim \int_{\eta < \xi} e^{-\eta^2}\, d\eta
\lesssim |\xi|^{-1} e^{-\xi^2}.
\end{equation*}
The same inequality holds when $\xi > 0$.  Further, $\widehat{f}$ is continuous in $[-1,1]$.  Thus, $f$ satisfies the hypotheses of
the $L^2(\R)$ version of the Hardy uncertainty principle and so is zero; this proves that $g = 0$ and $h(x) = h(0) e^{-x^2/4}$.

When $h(x,y)$ verifies \eqref{E: la condicion en Linf} in $\R^2$ and
\begin{equation}\label{E: un mŽtodo curioso}
h_\eta(x)=\frac 1{\sqrt{2\pi}} \int_\R e^{-iy\eta}h(x,y)\,dy,
\end{equation}
$h_\eta$ satisfies \eqref{E: la condicion en Linf} in $\R$, with  $\widehat{h_\eta}(\xi)=\widehat h(\xi,\eta)$ and the 1-dimensional Hardy uncertainty principle gives 
\begin{equation}\label{E: casicasi}
\widehat h(\xi,\eta)=e^{-\xi^2}h_2(\eta),
\end{equation}
for some $h_2:\R\longrightarrow\C$. Replacing $x$ by $y$ and $\xi$ by $\eta$ in \eqref{E: un mŽtodo curioso}, we get
\begin{equation}\label{E: casicasicai}
\widehat h(\xi,\eta)=e^{-\eta^2}h_1(\xi),
\end{equation}
for some $h_2:\R\longrightarrow\C$. Clearly, \eqref{E: casicasi} and \eqref{E: casicasicai} show that $\widehat h$ is a constant multiple of $e^{-\xi^2-\eta^2}$ and Theorem \ref{T: 4} follows.

See \cite[Theorem 4]{SiSu} for another real-variable reduction of the $L^\infty(\Rn)$ case, $n\ge 2$, to the $L^\infty(\R)$ case via the Radon transform.
\end{proof}

Let us recall, extended to $\R^n$,  the two variants $(A_1)$ and $(A_2)$ of Hardy uncertainty principle that we considered in the introduction:
\vspace{0.1cm}

$(A_1)$ \emph{If 
$h(x)=O(e^{-|x|^2/\beta^2})$, $\hat h(\xi)=O(e^{-4|\xi|^2/\alpha^2})$ and  $1/\alpha\beta>1/4$, then $h\equiv 0$. Also, if $1/\alpha\beta=1/4$, $h$ is a constant multiple of $e^{-|x|^2/\beta^2}$.}
\vspace{0.1cm}

$(A_2)$ \emph{ If $0<p,\,q\le \infty$, with at least one of them finite,
\begin{equation*}
\|e^{|x|^2/\beta^2}h\|_{L^p(\R^n)}+\|e^{4|\xi|^2/\alpha^2}\widehat h\|_{L^q(\R^n)}<+\infty,
\end{equation*}
and $1/\alpha\beta\ge  1/4$, then $h\equiv 0$.} 
\vspace{0.1cm}

In fact, Hardy proves in $\R$ the following stronger version  of $(A_2)$ \cite{Har}:
\vspace{0.1cm}

$(A_3)$ \emph{If $h:\Rn\longrightarrow\R$ verifies,
\begin{equation*}
h(x)=O((1+|x|^2)^{\frac k2}e^{-|x|^2/\beta^2}),\quad \hat h(\xi)=O((1+|\xi|^2)^{\frac k2}e^{-4|\xi|^2/\alpha^2}),
\end{equation*}
for some $k\ge 1$ and  $1/\alpha\beta>1/4$, then $h\equiv 0$. If $1/\alpha\beta=1/4$, $e^{|x|^2/\beta^2}h(x)$ is a polynomial of degree less than or equal to $k$.}
\vspace{0.2cm}

There is still another possible extension, namely \cite{bonamie}:
\vspace{0.2cm}

$(A_4)$ \emph{If  $\Phi \in \sS'(\R^n)$ (the space of tempered distributions) and 
\[
e^{|\cdot|^2/\alpha}\Phi \in \sS'(\R^n)
\quad\text{and}\quad
e^{\alpha|\cdot|^2/4}\widehat \Phi  \in \sS'(\R^n),
\]
then $\Phi = e^{-|\cdot|^2/\alpha}p$, where $p$ is a polynomial.}
\vspace{0.2cm}

Finally let us write Beurling-H\"ormander's condition in $\R^n$:
\vspace{0.2cm}

$(B)$ \emph{If  $h \in L^1(\R^n)$ and
\[
\int_{\R^n} \int_{\R^n} e^{|x\cdot \xi|} \,|h(x)| \, |\widehat h(\xi)| \,dx \,d\xi < \infty,
\]
then $h = 0$.}
\vspace{0.2cm}

We have the following result.

\begin{theorem}\label{T4}
Conditions ($A_j)$ with $j=1,2,3,4$ are all equivalent. Moreover condition $(B)$ implies $(A_j)$ for any $j.$
\end{theorem}
\begin{proof}

The proof of Theorem 2 also gives that $(A_1)$ and $(A_2)$ are equivalent. Let us see that $(A_1)$ is also equivalent to $(A_3)$.
Again it suffices to prove the result when $\alpha=\beta=2$,
\begin{equation}\label{E: con polinomio}
h(x)=O((1+|x|^2)^{\frac k2}e^{-|x|^2/4}),\quad \hat h(\xi)=O((1+|\xi|^2)^{\frac k2}e^{-|\xi|^2}),
\end{equation} 
and to establish first the case $n=1$. Define then $Th=f$ as in \eqref{E: una peque–a definicion}, $Th$ verifies the same as $h$ but with $k$ by replaced by $k-1$ and the result follows by induction on $k$. When $n>1$, the result follows by induction on $n\ge 1$. In particular, when $n=2$ and if $h$ verifies \eqref{E: con polinomio} in $\R^2$, define $h_\eta$ as in \eqref{E: un mŽtodo curioso}. Then, $h_\eta$ verifies \eqref{E: con polinomio} in $\R$ and
\begin{equation}\label{E: a lo que uno llegaria}
\widehat h(\xi,\eta)=\left(\sum_{p=0}^kc_p(\eta)\xi^p\right)e^{-\xi^2}=\left(\sum_{q=0}^kd_q(\xi)\eta^q\right)e^{-\eta^2},
\end{equation}
for some functions $c_p$, $d_q$, $p,q=1,\dots,k$. The later shows that 
\begin{equation*}
e^{\xi^2}d_j(\xi)=\sum_{i=0}^ka_{ij}\xi^i,\ \text{when} \ j=0,\dots,k\ \text{and for some}\ a_{ij}\in\C.
\end{equation*}
Thus, 
\begin{equation}\label{E: otra polinoio}
\widehat h(\xi,\eta)=\left(\sum_{i,j=0}^ka_{ij}\xi^i\eta^j\right)e^{-\xi^2-\eta^2},
\end{equation}
and the growth condition \eqref{E: con polinomio} implies that the polymonial in \eqref{E: otra polinoio} has degree less or equal than $k$. 

Let us now suppose that $(A_3)$ holds, and that $\Phi \in \sS'(\R)$ satisfies the hypotheses of $(A_4)$.
It will be convenient to define $\calP(\R)$ to be the space of all functions of the form $e^{-|\cdot|^2/2}p$, where $p$ is a polynomial on $\R$ with complex coefficients.

Write $\Upsilon_1$ and $\Upsilon_2$ for the tempered distributions $e^{|\cdot|^2/2}\Phi$ and $e^{|\cdot|^2/2} \widehat\Phi$.
Define
\[
\Psi = e^{-|\cdot|^2/4} (e^{-|\cdot|^2/2} * \Phi) .
\]
It is easy to check that $\widehat\Psi =  e^{-|\cdot|^2} * (e^{-|\cdot|^2/2} \widehat \Phi)$, and it follows at once that $\Psi$ and $\widehat\Psi$ are continuous functions.

In the following calculation involving two variables, $x$ and $y$, we write $\Phi_y$ to indicate that the distribution $\Phi$ acts on the functions of $y$ obtained by freezing the $x$ variable.
It is clear that $e^{-(x/2-y)^2} e^{y^2/2} = e^{x^2/4} e^{-(x-y)^2/2}$ for all $x, y \in \R$, and so
\begin{align*}
e^{x^2/2} \Psi(x)
&= e^{x^2/4} (e^{-|\cdot|^2/2} * \Phi) (x) \\
&= e^{x^2/4}  \Phi_y(  e^{-(x-y)^2/2} )\\
&= \Phi_y( e^{x^2/4} e^{-(x-y)^2/2} )\\
&= \Phi_y( e^{-(x/2-y)^2} e^{y^2/2})\\
&= [e^{|\cdot|^2/2} \Phi]_y (e^{-(x/2-y)^2})\\
&= e^{-|\cdot|^2} * \Upsilon_1 (x/2)
\end{align*}
for all $x \in \R$.
Similarly, $e^{\xi^2/2}  e^{-(\xi - \eta)^2} e^{-\eta^2/2} = e^{-2(\xi/2 - \eta)^2} e^{\eta^2/2}$ for all $\xi$ and $\eta$, so
\begin{align*}
e^{\xi^2/2} \widehat\Psi(\xi)
&= e^{\xi^2/2} [e^{-|\cdot|^2} * (e^{-|\cdot|^2/2} \widehat \Phi )](\xi) \\
&= e^{\xi^2/2} (e^{-|\cdot|^2/2} \widehat \Phi )_\eta ( e^{-(\xi - \eta)^2} ) \\
&= e^{\xi^2/2} (\widehat\Phi)_\eta ( e^{-(\xi - \eta)^2} e^{-\eta^2/2} ) \\
&= (\widehat\Phi)_\eta ( e^{\xi^2/2}  e^{-(\xi - \eta)^2} e^{-\eta^2/2} ) \\
&= (\widehat\Phi)_\eta ( e^{-2(\xi/2 - \eta)^2} e^{\eta^2/2} ) \\
&= (e^{|\cdot|^2/2} \widehat\Phi)_\eta ( e^{-2(\xi/2 - \eta)^2} ) \\
&= e^{-2|\cdot|^2} * \Upsilon_2 (\xi/2)
\end{align*}
for all $\xi \in \R$.
It is easy to see that, if $\Upsilon \in \sS'(\R)$ and $\gamma > 0$, then $e^{-\gamma |\cdot|^2} * \Upsilon$ is smooth and grows no faster than a polynomial, and hence $\Psi$ satisfies the hypotheses of $(A_3)$.
Then $\Psi \in \calP(\R)$, and hence $e^{-|\cdot|^2/2} * \Phi = e^{-|\cdot|^2/4} p$, where $p$ is a polynomial, so $e^{-|\cdot|^2/2} \widehat\Phi = e^{-|\cdot|^2} p_1$, where $p_1$ is a polynomial, whence $\widehat\Phi = e^{-|\cdot|^2/2} p_1$ and $\Phi = e^{-|\cdot|^2/2} p_2$, where $p_2$ is a polynomial.
Thus $(A_4)$ holds.
It is evident that if $(A_4)$ holds, then $(A_3)$ holds. Therefore we have proved that $(A_j)$ for $j=1,2,3,4$ are all equivalent.
We conclude our proof of the equivalence by showing that if $(A_3)$ does not hold, then for any $M \in \N$ there exists a nonzero smooth function $F$ such that
\[
|F(x)| \leq C (1 + |x|)^{-M} e^{-|x|^2/2}
\quad\text{and}\quad
|\widehat F(\xi)| \leq C (1 + |\xi|)^{-M} e^{-|\xi|^2/2}
\]
for all $x$ and $\xi$ in $\R$.
The existence of such a function clearly implies that  (B)  fails.

We consider as before $Th=f$ with $f$ given in \eqref{E: una peque–a definicion}. 
If $T^{k+1} h = e^{-|\cdot|^2/2}p$, where $p$ is a polynomial, for some $k \in \N$, then
\[
T^k h(x) - T^k h(0) e^{-|x|^2/2} = x p(x) e^{-|x|^2/2},
\]
for all $x \in \R$, and so $T^{k}h \in \calP(\R)$; iterating this if necessary, we deduce that $h \in \calP(\R)$.
In particular, if $T^kh =0$ for some $k$, then $h \in \calP(\R)$.

Now suppose that $h$ satisfies the hypotheses of $(A_3)$ but not the conclusion.
By iterating $T$, we see that $T^{N+1}h$ satisfies
\begin{equation}\label{star}
|T^{N+1} h(x)| \leq C (1 + |x|)^{-1} e^{-|x|^2/2}
\quad\text{and}\quad
|\widehat {T^{N+1}h}(\xi)| \leq C (1 + |\xi|)^{-1} e^{-|\xi|^2/2} 
\end{equation}
for all $x$ and $\xi$ in $\R$.
But if we iterate further, then we do not improve the estimate, since the sharpest estimate is
\[
|T^{N+1} h(x) - T^{N+1} h(0) e^{-|x|^2}| \leq C e^{-|x|^2/2}.
\]
However, if $h, Th, T^2h, \dots, T^kh$ are linearly dependent, that is, if
\[
\alpha_0 h + \alpha_1 Th + \dots +\alpha_k T^k h=0,
\]
then
\begin{align*}
\alpha_1 h(x)  + \dots + \alpha_k T^k h(x)
&= - \alpha_0 x h(x) + \alpha_1 h(0) e^{-|x|^2/2} + \dots + \alpha_k T^k h(0) e^{-|x|^2/2} \\
&= p_1(x) h(x) + q_1(x) e^{-|x|^2/2}
\end{align*}
for all $x$; unravelling this further, we find that $h$ is the product of a rational function and $e^{-|\cdot|^2/2}$.
Unless the denominator of the rational function is constant, $\widehat h$ does not decay as fast as functions in $\calP(\R)$.
It follows that we can exclude the possibility that $h$, $Th$, \dots, $T^kh$ are linearly dependent.
By taking linear combinations of the functions $T^{N+1}h$, $T^{N+2}h$, \dots, $T^{N+k}h$, where $k$ is a positive integer, we can therefore find nonzero functions $F$ that satisfy \eqref{star} and also satisfy $F(0) = F'(0) = \dots = F^{(M-1)} = 0$.
Then $x \mapsto x^{-M} F(x)$ is nonzero and continuous and satisfies
\begin{equation}\label{bad-fn}
|F(x)| \leq C (1 + |x|)^{-M} e^{-|x|^2/2}
\quad\text{and}\quad
|\widehat F(\xi)| \leq C (1 + |\xi|)^{-M} e^{-|\xi|^2/2} 
\end{equation}
for all $x$ and $\xi$.
\end{proof}
Observe that we have now shown that $(A_j)$ with $j=1,2,3,4$ are equivalent, and that $(B)$ implies all these. It is worth pointing out that $(B)$ is more general that the other variants, since it applies also to compactly supported functions whose Fourier transforms are of exponential type.

\end{section}
\vfill\eject
\pagebreak

\vfill\eject
%%%%%%%%%%%%%%%
%%%%%%%%%%%%%%%%%%%%%%%%%%
\end{document}